\documentclass[12pt]{article}

\usepackage{amsmath,amssymb,amsthm,amscd}
\usepackage[colorlinks, linkcolor=blue, anchorcolor=gree, citecolor=red]{hyperref}
\usepackage[numbers,sort&compress]{natbib}
\usepackage{graphicx}

\makeatletter

\newcommand{\Rmnum}[1]{\expandafter\@slowromancap\romannumeral #1@}
\makeatother

\begin{document}

\newcommand{\Cyc}{{\rm{Cyc}}}\newcommand{\diam}{{\rm{diam}}}
\newcommand{\Cay}{{\rm Cay}}
\newtheorem{thm}{Theorem}[section]
\newtheorem{pro}[thm]{Proposition}
\newtheorem{lem}[thm]{Lemma}
\newtheorem{exa}[thm]{Example}
\newtheorem{fac}[thm]{Fact}
\newtheorem{cor}[thm]{Corollary}
\newtheorem{constr}[thm]{Construction}
\theoremstyle{definition}
\newtheorem{ex}[thm]{Example}
\newtheorem{ob}[thm]{Observtion}
\newtheorem{remark}[thm]{Remark}
\newcounter{foo}[subsection]
\newcounter{fooo}[section]
\newtheorem{step}[foo]{Step}
\newtheorem{stepp}[fooo]{Step}
\newcommand{\bth}{\begin{thm}}
\renewcommand{\eth}{\end{thm}}
\newcommand{\bex}{\begin{ex}}
\newcommand{\eex}{\end{ex}}
\newcommand{\bre}{\begin{remark}}
\newcommand{\ere}{\end{remark}}

\newcommand{\bal}{\begin{aligned}}
\newcommand{\eal}{\end{aligned}}
\newcommand{\beq}{\begin{equation}}
\newcommand{\eeq}{\end{equation}}
\newcommand{\ben}{\begin{equation*}}
\newcommand{\een}{\end{equation*}}

\newcommand{\bpf}{\begin{proof}}
\newcommand{\epf}{\end{proof}}
\renewcommand{\thefootnote}{}
\newcommand{\sdim}{{\rm sdim}}

\def\beql#1{\begin{equation}\label{#1}}
\title{\Large\bf Perfect codes in quintic Cayley graphs on abelian groups}
\author{{Yuefeng Yang$^{1}$,\quad Xuanlong Ma$^{2}$,\quad Qing Zeng$^{3,4}$\footnote{Corresponding author.}
}\\[15pt]
{\small\em $^1$School of Science, China University of Geosciences, Beijing 100083, China}\\
{\small\em $^2$School of Science, Xi'an Shiyou University, Xi'an 710065, China}\\
{\small\em $^3$Department of Mathematical Sciences, Tsinghua University, Beijing 100084, China}\\
{\small\em $^4$Laboratory of Mathematics and Complex Systems {\rm (}Ministry of Education{\rm )},} \\
{\small\em School of Mathematical Sciences, Beijing Normal University, Beijing 100875, China}\\
}

 \date{}

\maketitle

\begin{abstract}
A subset $C$ of the vertex set of a graph $\Gamma$ is called a perfect code of $\Gamma$ if every vertex of $\Gamma$ is at distance no more than one to exactly one vertex in $C$. In this paper, we classify all connected quintic Cayley graphs on abelian groups that admit a perfect code, and determine completely all perfect codes of such graphs.
\end{abstract}



\medskip

\footnote{E-mail addresses: yangyf@cugb.edu.cn (Y. Yang), xuanlma@xsyu.edu.cn (X. Ma), qingz@mail.bnu.edu.cn (Q. Zeng).}
\section{Introduction}

A graph $\Gamma$ is a pair $(V\Gamma,E\Gamma)$ of vertex set $V\Gamma$ and edge set $E\Gamma$, where $E\Gamma$ is a subset of the set of $2$-element subsets of $V\Gamma$. In this paper, every group considered is finite, and
every graph considered is finite, simple and undirected.

Let $\Gamma=(V\Gamma,E\Gamma)$ be a graph. Two distinct vertices $x,y\in V\Gamma$ are called {\em adjacent} if the set $\{x,y\}\in E\Gamma$. The {\em distance} between two vertices $x,y\in V\Gamma$ is the length of a shortest path connecting $x$ and $y$ in $\Gamma$. For a positive integer $t$, a subset $C$ of $V\Gamma$ is called a {\em perfect $t$-code} \cite{Big,Kr} in $\Gamma$ if every vertex of $\Gamma$ is at distance no more than $t$ to exactly one vertex of $C$. In particular, $C$ is an independent set of $\Gamma$. A perfect 1-code is simply called a {\em perfect code}. In some references, a perfect code is also
called an {\em efficient dominating set} \cite{DYP,DYP2,DeS}
or {\em independent perfect domination set} \cite{Le}.

The notation of perfect $t$-codes in graphs were developed from the work in \cite{Big,Kr}, which in turn has a root in coding theory.  The Hamming distance between words of length $n$ over an alphabet of size $q\geq2$ is precisely the graph distance in the Hamming graph $H(n,q)$, and therefore perfect $t$-codes in $H(n,q)$ are exactly those in the classical setting \cite{MS} under the Hamming metric. It is well known that Hamming graphs are distance-transitive. This motivatied Biggs \cite{Big} to initiate an investigation of perfect codes in distance-transitive graphs. Since the seminal paper of Biggs \cite{Big} and the fundamental work of Delsarte \cite{DD}, perfect codes in distance-transitive graphs and, in general, in distance-regular graphs and association schemes have received considerable attention in the literature. See, for example, \cite{Ba,HS,Ze} and the survey papers \cite{He,Van}.

Perfect codes in Cayley graphs are another generalization of perfect codes in the
classical setting. This is so because $H(n,q)$ is the Cayley graph of the additive group $\mathbb{Z}_q^n$ with respect to the set of all elements of $\mathbb{Z}_q^n$ with exactly one nonzero coordinate. In general, for a group $G$ with identity element $e$ and an {\em inverse-closed} subset $S$ of $G\setminus\{e\}$ (that is, $S^{-1}:=\{x^{-1}\mid x\in S\}=S$), the {\em Cayley graph} $\Cay(G,S)$ of $G$ with respect to the {\em connection set} $S$ is defined to be the graph with vertex set $G$ such that $x,y\in G$ are adjacent whenever $yx^{-1}\in S$.
A perfect $t$-code in a Cayley graph can be viewed as a tiling of the underlying group by the balls of radius $t$ centered at the vertices of $C$ with respect to the graph distance. Moreover, if $X$ is a group with $q$ elements, $G=X^n$ is the direct product of $n$ copies of $X$, and $S$ consists of those elements of $G$ with exactly one nonidentity coordinate, then $\Cay(G,S)$ is isomorphic to $H(n,q)$, and each subgroup of $G$ is a group code. So the notion of perfect codes in Cayley graphs can be also thought as a
generalization of the concept of perfect group codes \cite{Van}. Therefore the study of perfect codes in Cayley graphs is highly related to both coding theory and group theory.

In the past a few years, perfect codes in Cayley graphs have attracted considerable attention. Lee proved that a Cayley graph on an abelian group has a perfect code if and only if it is a
covering graph of a complete graph \cite{Le}. Dejter and Serra gave a method for constructing countable families of Cayley graphs each one having a perfect code \cite{DeS}. Perfect codes in certain Cartesian products, where the
factors are Cayley graphs on an abelian group, are shown to exist by Mollard \cite{Mo}. In \cite{OPR}, cubic and quartic circulants (that is, Cayley graphs on cyclic groups) admitting a perfect code were classified. \c{C}al{\i}\c{s}kan, Miklavi\v{c} and \"{O}zkan characterized cubic and quartic Cayley
graphs on abelian groups that admit a perfect code \cite{CMO}. Kwon, Lee and Sohn gave necessary and sufficient conditions for the existence of perfect code of quintic circulants and classified these perfect codes \cite{KLS}. Caliskan, Miklavi\v{c} and \"{O}zkan classified the connected cubic Cayley graphs on
generalized dihedral groups which admit a perfect code \cite{CMO22}. For more about perfect codes in Cayley graphs, see for examples \cite{FHZ,Z15,HXZ18,CWZ,MWWZ,ZZ20,ZZ21}.

For two graphs $\Gamma$ and $\Sigma$, the {\em Cartesian product} $\Gamma\times\Sigma$ is the graph with vertex set $V\Gamma\times V\Sigma$ such that $(u_1,u_2)$ and $(v_1,v_2)$ are adjacent if and only if $\{u_1,v_1\}\in E\Gamma$ and $u_2=v_2$, or $u_1=v_1$ and $\{u_2,v_2\}\in E\Sigma$. For a positive integer $n$, denote by $K_{n}$ the complete graph ${\rm Cay}(\mathbb{Z}_{n},\mathbb{Z}_n\setminus\{0\})$.

For an integer $n$ and a prime $p$, let $\sigma_p(n)$ be the maximum nonnegative integer $i$ with $p^i\mid n$. In this paper, we study quintic Cayley
graphs on abelian groups that admit a perfect code, and recover the results in \cite{KLS}. The following theorem classifies all such graphs.

\begin{thm}\label{main}
A connected quintic Cayley graph on an abelian group admits a perfect code if and only if it is isomorphic to one of the following graphs:
\begin{itemize}
\item[{\rm (i)}] $\Gamma_{m,l,h}\times K_2$, where $\Gamma_{m,l,h}$ is from Construction \ref{construcution} with $\sigma_2(l)\neq0$ or $\sigma_2(m)>\sigma_2(l-ah)$;

\item[{\rm(ii)}] $\Gamma_{m,l,h}'$ in Construction \ref{cons-2}  with $\sigma_2(l)=0$ and $\sigma_2(m)=\sigma_2(l-ah)+1$, or $\sigma_2(h)\sigma_2(l)\neq0$ and $\sigma_2(m)\leq\sigma_2(l-ah)$;

\item[{\rm (iii)}] $\Gamma_{m,l,h}''$ in Construction \ref{cons-3} with $\sigma_2(h)\geq\sigma_2(m)>\sigma_2(l)\geq1$ or $\sigma_2(h)\geq\sigma_2(m)=\sigma_2(l)=1$.
\end{itemize}
Here, $0\leq h<m$, $0<l$, $6\mid m$ and $3\mid(l-ah)$ for some $a\in\{\pm1\}$.
\end{thm}

For the rest of this paper, we always assume that $G$ is an abelian group with identity $0$, written additively. For $a\in G$, let $o(a)$ denote the {\em order} of $a$, that is, the smallest positive integer $m$ such that $ma=0$. An element $a$ is called an {\em involution} if $o(a)=2$. To determine all perfect codes of a Cayley graph, we only need to consider all perfect codes containing identity since Cayley graphs are vertex transitive. The following theorem determines all perfect codes containing identity of a connected quintic Cayley graph on an abelian group.

\begin{thm}\label{main2}
With the notations in Theorem \ref{main}, let ${\rm Cay}(G,S)$ be connected, where $S=\{\pm s,\pm s',s_0\}$ with $o(s)=m$ and $o(s_0)=2$. Suppose that ${\rm Cay}(G,S)$ admits a perfect code. Then one of the following holds:
\begin{itemize}
\item[{\rm(i)}] if ${\rm Cay}(G,S)$ is isomorphic to one of the graphs in Theorem \ref{main} (i), (ii) with $\sigma_2(l)=0$, and (iii) with $\sigma_2(h)\geq\sigma_2(m)=\sigma_2(l)=1$, then all perfect codes containing identity are exactly
\begin{align}
\bigcup_{r=0}^{{\rm gcd}(l-ah,m)/3-1}D^{a}(r,t_r)~{\rm for~all}~t_r\in\{0,1\}~{\rm with}~r\geq1;\nonumber
\end{align}

\item[{\rm(ii)}] if ${\rm Cay}(G,S)$ is isomorphic to one of the graphs in Theorem \ref{main} (ii) with $\sigma_2(l)\neq0$, and (iii) with  $\sigma_2(h)\geq\sigma_2(m)>\sigma_2(l)\geq1$, then all perfect codes containing identity are exactly
\begin{align}
\bigcup_{r=0}^{{\rm gcd}(l-ah,m)/6-1}D^{a}(r,t_r)~{\rm for~all}~t_r\in\{0,1\}~{\rm with}~r\geq 1.\nonumber
\end{align}
\end{itemize}
Here, $D^{a}(r,t_r)=\{(3r+aj)s+js'+(j+t_r)s_0\mid j\in\mathbb{Z}\}$ and $t_0=0$.
\end{thm}

The paper is organized as follows. 
In Section 2, we construct some infinite families of  Cayley graphs on abelian groups admitting a perfect code. In Section 3, we give the proofs of Theorems \ref{main} and \ref{main2}.

\section{Constructions}

In this section, we construct three infinite families of quintic Cayley graphs on abelian groups admitting a perfect code. The main ideas for the constructions are taken from \cite[Construction 3]{YYF16}.

In the remainder of this section, we always assume that $m$ and $l$ are positive integers and $h$ is a nonnegative integer less than $m$. For a positive integer $n$, denote by $[n]$ the set $\{0,1,\ldots,n-1\}$.

\begin{constr}\label{construcution}
Let $\Gamma_{m,l,h}$ be the graph with the vertex set $\mathbb{Z}_m\times[l]$ whose edge set consists of $\{(a,b),(a+1,b)\}$, $\{(a,c),(a,c+1)\}$ and $\{(a,-1),(a-h,0)\}$, where $a\in\mathbb{Z}_m$, $b\in[l]$ and $c\in[l-1]$. See Figure \ref{fig:InformativeFigure}.
\end{constr}

For each integer $i$, let $\overline{i}$ denote the residue class $i+n\mathbb{Z}_n$, and $\hat{i}$ be the minimal nonnegative integer in $\overline{i}$. In $\mathbb{Z}_n$, if no confusion occurs, we write $i$ instead of $\overline{i}$.

\begin{remark}\label{h,l}
Let $j\in\mathbb{Z}$. Then there exist $n\in\mathbb{Z}$ and $r\in [l]$ such that $j=nl+r$. By Construction \ref{construcution}, we may use $(i,j)$ to denote the vertex $(\hat{i}-nh,r)$ of $V\Gamma_{m,l,h}$ for all $i\in\mathbb{Z}_m$.
\end{remark}

\begin{figure}[!h]
  \begin{center}
     \includegraphics{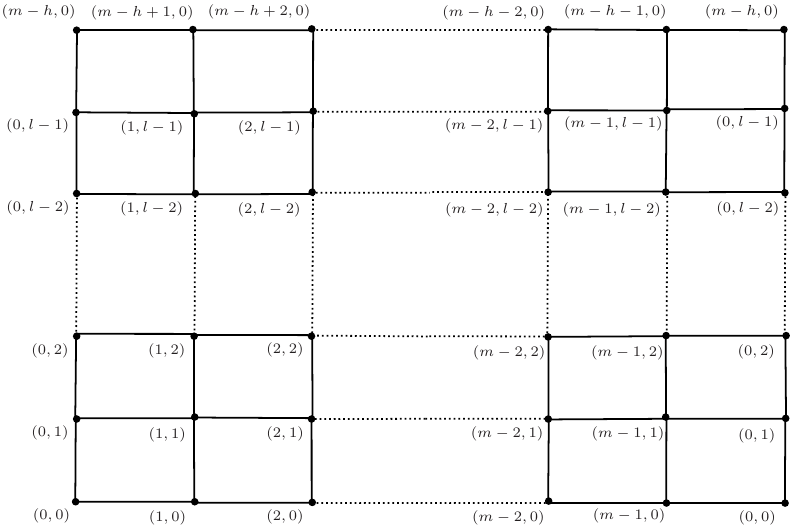}
  \end{center}
  \caption{\label{fig:InformativeFigure} The graph $\Gamma_{m,l,h}$.}
\end{figure}



For an integer $n$, let $P(n)$ be the set consisting of all prime factors of $n$. For a finite set $X$ consisting of primes, let $G_p$ be an abelian $p$-group for $p\in X$. Denote $G=\prod_{p\in X}G_p$. For an element $\alpha\in G$ and a prime $p\in X$, let $\alpha^{(p)}$ be the image of the projection from $G$ to $G_p$.

\begin{pro}\label{cayley}
Let $f_p(l,m,h)=\min\{\sigma_p(l),\sigma_p({\rm gcd}(m,h))\}$. The digraph $\Gamma_{m,l,h}$ is isomorphic to $\Cay(\prod_{p\in P(ml)}\mathbb{Z}_{p^{f_p(l,m,h)}}\times\mathbb{Z}_{p^{\sigma_p(ml)-f_p(l,m,h)}},\{\pm\alpha,\pm\beta\})$,
where
\begin{align}
\alpha^{(p)}&=\left\{
\begin{array}{ll}
(1,0), & \textrm{if}\ \sigma_p(l)>\sigma_p({\rm gcd}(m,h))=\sigma_p(m),\\
(1,l/{\rm gcd}(l,m-h)), & \textrm{if}\ \sigma_p(l)>\sigma_p({\rm gcd}(m,h))\neq\sigma_p(m),\\
(0,l/{\rm gcd}(l,m-h)), & \textrm{if}\ \sigma_p(l)\leq\sigma_p({\rm gcd}(m,h)),
\end{array} \right. \label{alpha}\\
\beta^{(p)}&=\left\{
\begin{array}{ll}
(0,1), & \textrm{if}\ \sigma_p(l)>\sigma_p({\rm gcd}(m,h))=\sigma_p(m),\\
(0,(m-h)/{\rm gcd}(l,m-h)), & \textrm{if}\ \sigma_p(l)>\sigma_p({\rm gcd}(m,h))\neq\sigma_p(m),\\
(1,(m-h)/{\rm gcd}(l,m-h)), & \textrm{if}\ \sigma_p(l)\leq\sigma_p({\rm gcd}(m,h)).
\end{array}\right.\label{beta}
\end{align}
\end{pro}
\begin{proof}
Let $\varphi$ be the mapping from $\mathbb{Z}_m\times[l]$ to $\prod_{p\in P(ml)}\mathbb{Z}_{p^{f_p(l,m,h)}}\times\mathbb{Z}_{p^{\sigma_p(ml)-f_p(l,m,h)}}$ such that
$\varphi(a,b)=\hat{a}\alpha+b\beta$. Note that $\varphi$ is well defined.

Now we show that $\varphi$ is injective. Let $\varphi(a,b)=\varphi(x,y)$ for $(a,b),(x,y)\in\mathbb{Z}_m\times[l]$. It follows that
\begin{align}
\hat{a}\alpha^{(p)}+b\beta^{(p)}=\hat{x}\alpha^{(p)}+y\beta^{(p)}~~{\rm for}~~p\in P(ml).\label{2.3}
\end{align}

We claim that $b\equiv y~({\rm mod}~p^{\sigma_p(l)})$ for all $p\in P(l)$. Suppose $\sigma_p(l)>\sigma_p({\rm gcd}(m,h))=\sigma_p(m)$ or $\sigma_p(l)\leq\sigma_p({\rm gcd}(m,h))$. By \eqref{alpha}--\eqref{2.3}, we have $b\equiv y~({\rm mod}~p^{\sigma_p(ml)-\sigma_p({\rm gcd}(m,h))})$ or $b\equiv y~({\rm mod}~p^{\sigma_p(l)})$. It follows that $b\equiv y~({\rm mod}~p^{\sigma_p(l)})$. Now suppose $\sigma_p(l)>\sigma_p({\rm gcd}(m,h))\neq\sigma_p(m)$.  Then $f_p(l,m,h)=\sigma_p({\rm gcd}(m,h))$ and $\sigma_p(h)<\sigma_p(m)$. By \eqref{alpha}--\eqref{2.3}, we have $\hat{a}\equiv \hat{x}~({\rm mod}~p^{\sigma_p({\rm gcd}(m,h))})$,
\begin{align}
\frac{(\hat{a}-\hat{x})l+(b-y)(m-h)}{{\rm gcd}(l,m-h)}\equiv0~({\rm mod}~p^{\sigma_p(ml)-\sigma_p({\rm gcd}(m,h))}).\label{eq-cayley-3}
\end{align}
Since $\sigma_p({\rm gcd}(m,h))=\sigma_p(h)$, one gets $\sigma_p({\rm gcd}(m,h))=\sigma_p(m-h)$. The fact $\hat{a}\equiv \hat{x}~({\rm mod}~p^{\sigma_p({\rm gcd}(m,h))})$ implies $\sigma_p({\rm gcd}(l,m-h))\leq\sigma_p(\hat{a}-\hat{x})$, and so $p^{\sigma_p(l)}\mid(\hat{a}-\hat{x})l/{\rm gcd}(l,m-h)$. \eqref{eq-cayley-3} implies $p^{\sigma_p(l)}\mid(b-y)(m-h)/{\rm gcd}(l,m-h)$. The fact $\sigma_p(l)>\sigma_p(m-h)$ implies $b\equiv y~({\rm mod}~p^{\sigma_p(l)})$. Thus, our claim is valid.

By the claim, we have $b\equiv y~({\rm mod}~l)$, and so $b=y$. By \eqref{2.3}, we have $\hat{a}\alpha=\hat{x}\alpha$. Pick $p\in P(m)$. If $\sigma_p(l)>\sigma_p({\rm gcd}(m,h))=\sigma_p(m)$, then $f_p(l,m,h)=\sigma_p(m)$, and so $\hat{a}\equiv\hat{x}~({\rm mod}~p^{\sigma_p(m)})$ from \eqref{alpha}. Suppose $\sigma_p(l)\leq\sigma_p({\rm gcd}(m,h))$. Then $f_p(l,m,h)=\sigma_p(l)$ and $p\nmid l/{\rm gcd}(l,m-h))$. By \eqref{alpha}, one has $p^{\sigma_p(m)}\mid (\hat{a}-\hat{x})l/{\rm gcd}(l,m-h)$, and so $p^{\sigma_p(m)}\mid \hat{a}-\hat{x}$. Suppose $\sigma_p(l)>\sigma_p({\rm gcd}(m,h))\neq\sigma_p(m)$. Then $\sigma_p(h)<\sigma_p(m)$ and $\sigma_p(m-h)=\sigma_p({\rm gcd}(m,h))$. It follows that $f_p(l,m,h)=\sigma_{p}({\rm gcd}(m,h))=\sigma_p({\rm gcd}(l,m-h))$. By \eqref{alpha} again, we have $p^{\sigma_p(ml)-\sigma_p({\rm gcd}(m,h))}\mid (\hat{a}-\hat{x})l/{\rm gcd}(l,m-h)$, and so $p^{\sigma_p(m)}\mid \hat{a}-\hat{x}$. We conclude that $p^{\sigma_p(m)}\mid \hat{a}-\hat{x}$ for all $p\in P(m)$. It follows that $\hat{a}\equiv\hat{x}~({\rm mod}~m)$, and so $a=x$.  Thus, $\varphi$ is injective.

Since $|V\Gamma_{m,l,h}|=\prod_{p\in P(ml)}|\mathbb{Z}_{p^{f_p(l,m,h)}}\times\mathbb{Z}_{p^{\sigma_p(ml)-f_p(l,m,h)}}|$, $\varphi$ is a bijection. One can verify that
$(x_{1},y_{1})$ and $(x_{2},y_{2})$ are adjacent if and only if $\varphi(x_{1},y_{1})$ and $\varphi(x_{2},y_{2})$ are adjacent. Thus, $\varphi$ is an isomorphism.
\end{proof}

Denote $\gamma(n)=(3+(-1)^n)/2$ for each integer $n$.

\begin{pro}\label{(iv)}
Let $6\mid m$ and $3\mid(l-ah)$ for some $a\in\{\pm1\}$. Suppose $\sigma_2(l)\neq0$ or $\sigma_2(m)>\sigma_2(l-ah)$. Then $\cup_{r\in[{\rm gcd}(l-ah,m)/3]}C^{a}(r,t_r)$ is a perfect code of $\Gamma_{m,l,h}\times K_2$ for $t_r\in\{0,1\}$, where $C^{a}(r,t_r)=\{(3r+aj,j,j+t_r)\mid j\in\mathbb{Z}\}$.
\end{pro}
\begin{proof}
For fixed integers $h,l,m$ and $a$, let $b={\rm gcd}(l-ah,m)$. Then $3\mid b$. Denote $C=\cup_{r\in[b/3]}C^{a}(r,t_r)$ for fixed $t_r$. If $2\mid l$, then $ah\equiv l~({\rm mod}~b)$; if $2\nmid l$, then $ah\equiv l+b~({\rm mod}~2b)$ since $\sigma_2(m)>\sigma_2(l-ah)$. Then $ah\equiv l+b~({\rm mod}~2b/\gamma(l))$.

By Remark \ref{h,l}, $l$ is the minimal positive integer such that $(i,l,k)\in\{(i',0,k')\mid (i',k')\in\mathbb{Z}_m\times\mathbb{Z}_2\}$ for each $(i,k)\in\mathbb{Z}_m\times\mathbb{Z}_2$. Note that $(3r+al-h,0,l+t_r)=(3r+al,l,l+t_r),(3r,0,t_r)\in C^a(r,t_r)$. Since the order of $al-h$ in $\mathbb{Z}_m$ is $m/b$, from Remark \ref{h,l}, we have $|C^a(r,t_r)|=l\cdot m/b$ for all $r\in[b/3]$.

Suppose $C^a(r,t_r)\cap C^a(s,t_{s})\neq\emptyset$ for distinct $r,s\in[b/3]$. Since $\Gamma_{m,l,h}\times K_2$ is vertex transitive from Proposition \ref{cayley}, we may assume $r=0$ and $(0,0,t_0)=(3s+anl,nl,nl+t_s)$ for some $n\in\mathbb{Z}$. By Remark \ref{h,l}, one has $(0,0,t_0)=(3s+n(al-h),0,nl+t_s)$. It follows that $3s\in(al-h)\mathbb{Z}+m\mathbb{Z}=b\mathbb{Z}$, contrary to the fact that $s\in[b/3]\setminus\{0\}$. Thus, $|C|=ml/3$.

We claim that $(i,j,k)\in C$ if and only if there exists $r\in[b/3]$ such that
\begin{align}
a(\hat{i}-3r)\equiv b(\hat{k}-t_r)+(b+1)j~({\rm mod}~2b/\gamma(l)),\quad \hat{k}\equiv j+t_r~({\rm mod}~\gamma(l)).\label{2.2}
\end{align}
The proof of necessity is trivial. Now suppose that \eqref{2.2} holds for some $r\in[b/3]$ and $(i,j,k)\in\mathbb{Z}_m\times[l]\times\mathbb{Z}_2$. If $2\mid l$, then $a(\hat{i}-3r)\equiv j~({\rm mod}~b)$ and $\hat{k}\equiv j+t_r~({\rm mod}~2)$, which imply that there are exactly $m/b$ elements $(i,k)\in\mathbb{Z}_m\times\mathbb{Z}_2$ satisfying \eqref{2.2} for each $j\in[l]$ and $r\in[b/3]$. Now we consider the case $2\nmid l$. Note that \eqref{2.2} is equivalent to $a(\hat{i}-3r)\equiv b(\hat{k}-t_r)+(b+1)j~({\rm mod}~2b)$. It follows that $\hat{i}\equiv3r+aj~({\rm mod}~b)$. If $\hat{i}\equiv3r+aj~({\rm mod}~2b)$, then $\hat{k}\equiv j+t_r~({\rm mod}~2)$; if $\hat{i}\equiv3r+aj+b~({\rm mod}~2b)$, then $\hat{k}\equiv j+t_r+1~({\rm mod}~2)$. This implies that there are exactly $m/b$ elements $(i,k)\in\mathbb{Z}_m\times\mathbb{Z}_2$ satisfying \eqref{2.2} for each $r\in[b/3]$ and $j\in[l]$. Therefore, there are exactly $ml/3$ elements in $\mathbb{Z}_m\times[l]\times\mathbb{Z}_2$ satisfying \eqref{2.2}. Since $|C|=ml/3$, from the necessity, our claim is valid.

One can verify that $C$ is an independent set. Let $(i,j,k)\in\mathbb{Z}_m\times[l]\times\mathbb{Z}_2\setminus C$. Then \eqref{2.2} does not hold for $r\in[b/3]$. If $2\mid l$, then $a(\hat{i}-3r)\equiv j+c~({\rm mod}~b)$ and $\hat{k}\equiv j+t_r+d~({\rm mod}~2)$ for some $r\in[b/3]$ with $(c,d)\in\{(\pm1,0),(b,1),(\pm(b+1),1)\}$; if $2\nmid l$, then $a(\hat{i}-3r)\equiv b(\hat{k}-t_r)+(b+1)j+c~({\rm mod}~2b)$ and $\hat{k}\equiv j+t_r+d~({\rm mod}~1)$ for some $r\in[b/3]$ with $(c,d)\in\{(\pm1,0),(b,1),(\pm(b+1),1)\}$. Thus,
\begin{align}
a(\hat{i}-3r)\equiv b(\hat{k}-t_r)+(b+1)j+c~({\rm mod}~2b/\gamma(l)),~\hat{k}\equiv j+t_r+d~({\rm mod}~\gamma(l))\label{2.2-1}
\end{align}
for some $r\in[b/3]$ with $(c,d)\in\{(\pm1,0),(b,1),(\pm(b+1),1)\}$.

By \eqref{2.2-1} and the claim, if $(c,d)=(\pm1,0)$, then $(i-1,j,k)$ or $(i+1,j,k)$ is the unique neighbor of $(i,j,k)$ in $C$; if $(c,d)=(b,1)$, then $(i,j,k+1)$ is the unique neighbor of $(i,j,k)$ in $C$.

Now we consider the case $(c,d)=(\pm(b+1),1)$. \eqref{2.2-1} implies $a(\hat{i}-3r)\equiv b(\hat{k}-t_r)+(b+1)(j\pm1)~({\rm mod}~2b/\gamma(l))$ and $\hat{k}\equiv j+1+t_r~({\rm mod}~\gamma(l))$. If $c=-b-1$ and $j\neq0$, or $c=b+1$ and $j\neq l-1$, from  the claim, then $(i,j-1,k)$ or $(i,j+1,k)$ is the unique neighbor of $(i,j,k)$ in $C$.

Suppose $c=-b-1$ and $j=0$. By \eqref{2.2-1}, we have $a(\hat{i}-3r)\equiv b(\hat{k}-t_r)-b-1~({\rm mod}~2b/\gamma(l))$ and $\hat{k}\equiv t_r+1~({\rm mod}~\gamma(l))$. Since $ah\equiv l+b~({\rm mod}~2b/\gamma(l))$, we have
\begin{align}
a(\hat{i}+h-3r)\equiv  b(\hat{k}-t_r)+l-1\equiv b(\hat{k}-t_r)+(b+1)(l-1)~({\rm mod}~2b/\gamma(l)).\nonumber
\end{align}
Since $\hat{k}\equiv l-1+t_r~({\rm mod}~\gamma(l))$, from the claim, $(i+h,l-1,k)$ is the unique neighbor of $(i,0,k)$ in $C$.

Suppose $c=b+1$ and $j=l-1$. By \eqref{2.2-1}, we have $a(\hat{i}-3r)\equiv b(\hat{k}-t_r)+(b+1)l\equiv b(\hat{k}-t_r)+b+l~({\rm mod}~2b/\gamma(l))$ and $\hat{k}\equiv t_r+1~({\rm mod}~\gamma(l))$. Since $ah\equiv l+b~({\rm mod}~2b/\gamma(l))$, we have $a(\hat{i}-h-3r)\equiv b(\hat{k}-t_r)~({\rm mod}~2b/\gamma(l))$. The claim implies that $(i-h,0,k)$ is the unique neighbor of $(i,l-1,k)$ in $C$.

This completes the proof of this proposition.
\end{proof}

\begin{constr}\label{cons-2}
Let $2\mid m$. We define the graph $\Gamma_{m,l,h}'$ with vertex set $V\Gamma_{m,l,h}$ such that $(a,b)$ and $(a',b')$ are adjacent if and only if $\{(a,b),(a',b')\}\in E\Gamma_{m,l,h}$, or $a'=a+m/2$ and $b'=b$.
\end{constr}

The following corollary is immediate from Proposition \ref{cayley}.

\begin{cor}\label{cayley-2}
With the notations in Proposition \ref{cayley}, suppose $2\mid m$. Then $\Gamma_{m,l,h}'$ is isomorphic to $\Cay(\prod_{p\in P(ml)}\mathbb{Z}_{p^{f_p(l,m,h)}}\times\mathbb{Z}_{p^{\sigma_p(ml)-f_p(l,m,h)}},\{\pm\alpha,\pm\beta,m\alpha/2\})$.
\end{cor}

\begin{pro}\label{(v)}
Let $6\mid m$ and $3\mid(l-ah)$ for some $a\in\{\pm1\}$. Suppose $\sigma_2(l)=0$ and $\sigma_2(m)=\sigma_2(l-ah)+1$, or $\sigma_2(h)\sigma_2(l)\neq0$ and $\sigma_2(m)\leq\sigma_2(l-ah)$. Then $\cup_{r\in[{\rm gcd}(l-ah,m)/3\gamma(l)]}C^{a}(r,t_r)$ is a perfect code of $\Gamma_{m,l,h}'$ for $t_r\in\{0,1\}$, where
\begin{align}
C^{a}(r,t_r)=\{(3r+aj+(j+t_r)m/2,j)\mid j\in\mathbb{Z}\}.\nonumber
\end{align}
\end{pro}
\begin{proof}
For fixed $h,l,m$ and $a$, let $b={\rm gcd}(l-ah,m)$. Then $3\mid b$. Denote $C=\cup_{r\in[b/3\gamma(l)]}C^{a}(r,t_r)$ for fixed $r\in[b/3\gamma(l)]$.

If $\sigma_2(l)=0$, then $\sigma_2(m)=\sigma_2(l-ah)+1$, which implies $m/2\equiv l-ah\equiv b~({\rm mod}~2b)$; if $\sigma_2(l)\neq0$, then $\sigma_2(m)\leq\sigma_2(l-ah)$, which implies $m/2\equiv b/2~({\rm mod}~b)$ and $l\equiv ah~({\rm mod}~b)$. Thus, $m/2\equiv b/\gamma(l)~({\rm mod}~2b/\gamma(l))$ and $l-ah\equiv b~({\rm mod}~2b/\gamma(l))$.

By Remark \ref{h,l} and Construction \ref{cons-2}, $l$ is the minimal positive integer such that $(i,l)\in\{(i',0)\mid i'\in\mathbb{Z}_m\}$ for each $i\in\mathbb{Z}_m$. Note that $(3r+al-h+(\gamma(l)+t_r)m/2,0)=(3r+al+(l+t_r)m/2,l),(3r+t_rm/2,0)\in C^a(r,t_r)$ for $r\in[b/3\gamma(l)]$. Observe that the order of $al-h+\gamma(l)m/2$ in $\mathbb{Z}_m$ is $m/{\rm gcd}(m,al-h+\gamma(l)m/2)$. If $2\mid l$, then $m/{\rm gcd}(m,al-h+\gamma(l)m/2)=m/b$; if $2\nmid l$, then $\sigma_2(m)=\sigma_2(l-ah)+1$, and so $\sigma_2(al-h+m/2)\geq\sigma_2(m)=\sigma_2(b)+1$, which implies $m/{\rm gcd}(m,al-h+m/2)=m/2b$. Then $|C^a(r,t_r)|=lm/{\rm gcd}(m,al-h+\gamma(l)m/2)=\gamma(l)ml/2b$.

Suppose $C^a(r,t_r)\cap C^a(s,t_{s})\neq\emptyset$ for distinct $r,s\in[b/3\gamma(l)]$. Since $\Gamma_{m,l,h}'$ is vertex transitive from Corollary \ref{cayley-2}, we may assume $r=0$ and $(t_0m/2,0)=(3s+anl+(nl+t_s)m/2,nl)$ for some $n\in\mathbb{Z}$. By Remark \ref{h,l}, we have $(t_0m/2,0)=(3s+n(al-h)+(nl+t_s)m/2,0)$, which implies $3s\in(al-h)\mathbb{Z}+(m/2)\mathbb{Z}=(b/\gamma(l))\mathbb{Z}$, contrary to the fact that $s\in[b/3\gamma(l)]\setminus\{0\}$. Thus, $|C|=ml/6$.

We claim that $(i,j)\in C$ if and only if there exists $r\in[b/3\gamma(l)]$ such that
\begin{align}
\hat{i}-3r-b t_r/\gamma(l)\equiv(b/\gamma(l)+a)j~({\rm mod}~2b/\gamma(l)).\label{2.6}
\end{align}
Since $m/2\equiv b/\gamma(l)~({\rm mod}~2b/\gamma(l))$, the proof of necessity is valid. Now suppose that \eqref{2.6} holds for some $r\in[b/3\gamma(l)]$ and $(i,j)\in\mathbb{Z}_m\times[l]$. Then $\hat{i}\equiv3r+aj~({\rm mod}~b/\gamma(l))$. If $\hat{i}\equiv3r+aj~({\rm mod}~2b/\gamma(l))$, then $j\equiv t_r~({\rm mod}~2)$; if $\hat{i}-3r-aj\equiv b/\gamma(l)~({\rm mod}~2b/\gamma(l))$, then $j\equiv t_r+1~({\rm mod}~2)$. Hence, there are exactly $m\gamma(l)/2b$ elements $i\in\mathbb{Z}_m$ satisfying \eqref{2.6} for each $r\in[b/3\gamma(l)]$ and $j\in[l]$.  Therefore, there are exactly $ml/6$ elements in $\mathbb{Z}_m\times[l]$ satisfying \eqref{2.6}. Since $|C|=ml/6$, from the necessity, our claim is valid.

Since $m/2\equiv b/\gamma(l)~({\rm mod}~2b/\gamma(l))$, $C$ is an independent set. Let $(i,j)\in\mathbb{Z}_m\times[l]\setminus C$. Then \eqref{2.6} does not hold for $r\in[b/3\gamma(l)]$. It follows that
\begin{align}
\hat{i}-3r-b t_r/\gamma(l)\equiv(b/\gamma(l)+a)j+c~({\rm mod}~2b/\gamma(l))\label{2.6-1}
\end{align}
for some $r\in[b/3\gamma(l)]$ with $c\in\{\pm1,b/\gamma(l),\pm(b/\gamma(l)+a)\}$.

By \eqref{2.6-1} and the claim, if $c=\pm1$, then $(i-1,j)$ or $(i+1,j)$ is the unique neighbor of $(i,j)$ in $C$; if $c=b/\gamma(l)$, then $(i+m/2,j)$ is the unique neighbor of $(i,j)$ in $C$ since $m/2\equiv b/\gamma(l)~({\rm mod}~2b/\gamma(l))$.

Now we consider the case $c=\pm(b/\gamma(l)+a)$. If $c=-(b/\gamma(l)+a)$ and $j\neq0$, or $c=b/\gamma(l)+a$ and $j\neq l-1$, by \eqref{2.6-1}, then $\hat{i}-3r-b t_r/\gamma(l)\equiv(b/\gamma(l)+a)(j\mp1)~({\rm mod}~2b/\gamma(l))$, which implies that $(i,j-1)$ or $(i,j+1)$ is the unique neighbor of $(i,j)$ in $C$ from the claim.

Suppose $c=-(b/\gamma(l)+a)$ and $j=0$. By \eqref{2.6-1}, we have $\hat{i}-3r-bt_r/\gamma(l)\equiv-(b/\gamma(l)+a)~({\rm mod}~2b/\gamma(l))$. Since $l-ah\equiv b\equiv bl/\gamma(l)~({\rm mod}~2b/\gamma(l))$, one gets $\hat{i}+h-3r-bt_r/\gamma(l)\equiv (a+b/\gamma(l))l-(b/\gamma(l)+a)~({\rm mod}~2b/\gamma(l))$. The claim implies that $(i+h,l-1)$ is the unique neighbor of $(i,0)$ in $C$.

Suppose $c=b/\gamma(l)+a$ and $j=l-1$. By \eqref{2.6-1}, we have $\hat{i}-3r-b t_r/\gamma(l)\equiv bl/\gamma(l)+al~({\rm mod}~2b/\gamma(l))$. Since $l-ah\equiv bl/\gamma(l)~({\rm mod}~2b/\gamma(l))$, one gets
$\hat{i}-h-3r-bt_r/\gamma(l)\equiv0~({\rm mod}~2b/\gamma(l))$. The claim implies that $(i-h,0)$ is the unique neighbor of $(i,l-1)$ in $C$.

This completes the proof of this proposition.
\end{proof}

\begin{constr}\label{cons-3}
Let $\sigma_2(h)\geq\sigma_2(m)\geq1$ and $\sigma_2(l)\geq1$. We define the graph $\Gamma_{m,l,h}''$ with vertex set $V\Gamma_{m,l,h}$ such that $(a,b)$ and $(a',b')$ are adjacent if and only if $\{(a,b),(a',b')\}\in E\Gamma_{m,l,h}$, or $a'=a+(m+h-{\rm lcm}(h,m))/2$ and $b'=b+l/2$.
\end{constr}

The following corollary is also immediate from Proposition \ref{cayley}.

\begin{cor}\label{cayley-3}
With the notations in Proposition \ref{cayley}, suppose $\sigma_2(h)\geq\sigma_2(m)\geq1$ and $\sigma_2(l)\geq1$. Then $\Gamma_{m,l,h}''$ is isomorphic to $$\Cay(\prod_{p\in P(ml)}\mathbb{Z}_{p^{f_p(l,m,h)}}\times\mathbb{Z}_{p^{\sigma_p(ml)-f_p(l,m,h)}},\{\pm\alpha,\pm\beta,((m+h-{\rm lcm}(h,m))\alpha+l\beta)/2\}).$$
\end{cor}

\begin{pro}\label{(vi)}
Let $6\mid m$ and $3\mid(l-ah)$ for some $a\in\{\pm1\}$. Suppose $\sigma_2(h)\geq\sigma_2(m)=\sigma_2(l)=1$ or $\sigma_2(h)\geq\sigma_2(m)>\sigma_2(l)\geq1$. Then $\cup_{r\in[{\rm gcd}(l-ah,m)/3\gamma(m/2)]}C^{a}(r,t_r)$ is a perfect code of $\Gamma_{m,l,h}''$ for $t_r\in\{0,1\}$, where $$C^{a}(r,t_r)=\{(3r+aj+(j+t_r)(m+h-{\rm lcm}(h,m))/2,j+(j+t_r)l/2)\mid j\in\mathbb{Z}\}.$$
\end{pro}
\begin{proof}
For fixed integers $h,l,m$ and $a$, let $b={\rm gcd}(l-ah,m)$. Then $3\mid b$. Denote $C=\cup_{r\in[b/3\gamma(m/2)]}C^{a}(r,t_r)$ for fixed $t_r$.

\begin{step}\label{step-0}
If $\sigma_2(l)=1$, then ${\rm gcd}(m,(al-h+m+{\rm lcm}(h,m))/2)=b/\gamma(m/2)$.
\end{step}

Suppose $\sigma_2(h)>\sigma_2(m)$. Then $2m\mid {\rm lcm}(h,m)$. It follows that ${\rm gcd}(m,(al-h+m+{\rm lcm}(h,m))/2)={\rm gcd}(m,(al-h+m)/2)$. If $\sigma_2(m)=1$, then $\sigma_2(al-h-m)>1$, and so  ${\rm gcd}(m,(al-h+m)/2)=b$. If $\sigma_2(m)>1$, then $\sigma_2(al-h-m)=1$, and so ${\rm gcd}(m,(al-h+m)/2)=b/2$. Then ${\rm gcd}(m,(al-h+m)/2)=b/\gamma(m/2)$.

Suppose $\sigma_2(h)=\sigma_2(m)$. Then $2m\mid {\rm lcm}(h,m)+m$. It follows that ${\rm gcd}(m,(al-h+m+{\rm lcm}(h,m))/2)={\rm gcd}(m,(al-h)/2)$. If $\sigma_2(m)=1$, then $\sigma_2(al-h)>1$, and so ${\rm gcd}(m,(al-h)/2)=b$. If $\sigma_2(m)>1$, then $\sigma_2(al-h)=1$, and so ${\rm gcd}(m,(al-h)/2)=b/2$. Then ${\rm gcd}(m,(al-h)/2)=b/\gamma(m/2)$.

\begin{step}\label{step-1}
$|C|=ml/6$.
\end{step}

{\bf Case 1.} $\sigma_2(l)=1$.

Since $2\mid l/2+1$, from Remark \ref{h,l}, $l/2$ is the minimal positive integer $j$ with $(i,j+(j+t_r)l/2)\in\{(i',t_rl/2)\mid i'\in\mathbb{Z}_m\}$ for $i\in\mathbb{Z}_m$. Note that $(3r+(al-h+m-{\rm lcm}(h,m))/2+t_r(m+h-{\rm lcm}(h,m))/2,t_rl/2)=(3r+al/2+(l/2+t_r)(m+h-{\rm lcm}(h,m))/2,l/2+(l/2+t_r)l/2)\in C^a(r,t_r)$. Note that $(3r+t_r(m+h-{\rm lcm}(h,m))/2,t_rl/2)\in C^a(r,t_r)$. Since the order of $(al-h+m-{\rm lcm}(h,m))/2$ in $\mathbb{Z}_m$ is $m/{\rm gcd}(m,(al-h+m+{\rm lcm}(h,m))/2)$, from Step \ref{step-0}, we have $|C^a(r,t_r)|=(l/2)\cdot m\gamma(m/2)/b$.

Suppose $C^a(r,t_r)\cap C^a(s,t_{s})\neq\emptyset$ for distinct $r,s\in[b/3\gamma(m/2)]$. Since $\Gamma_{m,l,h}''$ is vertex transitive from Corollary \ref{cayley-3}, we may assume $r=0$ and $(t_0(m+h-{\rm lcm}(h,m))/2,t_0l/2)=(3s+anl/2+(nl/2+t_s)(m+h-{\rm lcm}(h,m))/2,nl/2+(nl/2+t_s)l/2)$ for some $n\in\mathbb{Z}$. This implies $(t_0(m+h-{\rm lcm}(h,m))/2,t_0l/2)=(3s+n(al-h+m-{\rm lcm}(h,m))/2+t_s(m+h-{\rm lcm}(h,m))/2,t_sl/2)$. It follows that $t_s=t_0$, and so $3s\in ((al-h+m+{\rm lcm}(h,m))/2)\mathbb{Z}+m\mathbb{Z}$. By Step \ref{step-0}, one gets $3s\in b/\gamma(m/2)\mathbb{Z}$, contrary to the fact that $s\in [b/3\gamma(m/2)]\setminus\{0\}$. Thus, $|C|=ml/6$.

{\bf Case 2.} $\sigma_2(l)>1$.

By Remark \ref{h,l}, $l$ is the minimal positive integer $j$ with $(i,j+(j+t_r)l/2)\in\{(i',t_rl/2)\mid i'\in\mathbb{Z}_m\}$ for each $i\in\mathbb{Z}_m$. Since $\sigma_2(h)\geq\sigma_2(m)>1$, we have $(3r+al-h+t_r(m+h-{\rm lcm}(h,m))/2,t_rl/2)=(3r+al+(l+t_r)(m+h-{\rm lcm}(h,m))/2,l+(l+t_r)l/2)\in C^a(r,t_r)$ for $r\in[b/6]$. Since $(3r+t_r(m+h-{\rm lcm}(h,m))/2,t_rl/2)\in C^a(r,t_r)$ and the order of $al-h$ in $\mathbb{Z}_m$ is $m/{\rm gcd}(m,al-h)$, one gets $|C^a(r,t_r)|=ml/b$.

Note that $\sigma_2(m)>1$. Suppose $C^a(r,t_r)\cap C^a(s,t_{s})\neq\emptyset$ for distinct $r,s\in[b/6]$. Since $\Gamma_{m,l,h}''$ is vertex transitive from Corollary \ref{cayley-3}, we may assume $r=0$ and $(t_0(m+h-{\rm lcm}(h,m))/2,t_0l/2)=(3s+anl+(nl+t_s)(m+h-{\rm lcm}(h,m))/2,nl+(nl+t_s)l/2)$ for some $n\in\mathbb{Z}$. It follows that  $(t_0(m+h-{\rm lcm}(h,m))/2,t_0l/2)=(3s+n(al-h)+t_s(m+h-{\rm lcm}(h,m))/2,t_sl/2)$. Then, $t_s=t_0$, and so $3s\in(al-h)\mathbb{Z}+m\mathbb{Z}=b\mathbb{Z}$, contrary to the fact that $s\in[b/6]\setminus\{0\}$. Thus, $|C|=ml/6$.

\begin{step}\label{step-2}
$(m+h-al)/2+b/\gamma(m/2)\equiv {\rm lcm}(h,m)/2~({\rm  mod}~b\gamma(l/2)/\gamma(m/2))$.
\end{step}

Suppose $\sigma_2(h)\geq\sigma_2(m)=\sigma_2(l)=1$. It suffices to show that $(m+h-al)/2\equiv {\rm lcm}(h,m)/2~({\rm  mod}~b)$. Note that $\sigma_2(b)=\sigma_2(m)$. It follows that $m/2\equiv b/2~({\rm mod}~b)$ and ${\rm lcm}(h,m)/2\equiv\gamma(h/2)b/2~({\rm mod}~b)$. If $\sigma_2(m)=\sigma_2(h)$, then $\sigma_2(b)<\sigma_2(ah-l)$; if $\sigma_2(m)<\sigma_2(h)$, then $\sigma_2(b)=\sigma_2(ah-l)$. It follows that $(ah-l)/2\equiv b/\gamma(h/2)~({\rm mod}~b)$. Thus, $(m+h-al)/2\equiv{\rm lcm}(h,m)/2~({\rm  mod}~b)$.

Suppose $\sigma_2(h)\geq\sigma_2(m)>\sigma_2(l)\geq1$. It suffices to show that $(m+h-al)/2+b/2\equiv {\rm lcm}(h,m)/2~({\rm  mod}~b\gamma(l/2)/2)$. Note that $\sigma_2(b)=\sigma_2(l)<\sigma_2(m)$. It follows that $m/2\equiv{\rm lcm}(h,m)/2\equiv0~({\rm mod}~b\gamma(l/2)/2)$ and $(ah-l)/2\equiv b/2~({\rm mod}~b\gamma(l/2)/2)$. Thus, $(m+h-al)/2+b/2\equiv{\rm lcm}(h,m)/2~({\rm  mod}~b\gamma(l/2)/2)$.

This completes the proof of this step.

\begin{step}\label{step-3}
$(i,j)\in C$ if and only if there exists $r\in[b/3\gamma(m/2)]$ such that
\begin{align}
a(\hat{i}-3r)\equiv j+b(j+t_r)/\gamma(m/2)~({\rm mod}~b\gamma(l/2)/\gamma(m/2)),~ j\equiv t_r~({\rm mod}~2/\gamma(l/2)).\label{2.9}
\end{align}
\end{step}

Suppose $(i,j)=(3r+aj'+(j'+t_r)(m+h-{\rm lcm}(h,m))/2,j'+(j'+t_r)l/2)\in C$ for some $j'\in\mathbb{Z}$ and $r\in[b/3\gamma(m/2)]$. Then $j\equiv t_r~({\rm mod}~2/\gamma(l/2))$. Step \ref{step-2} implies
\begin{align}
a(\hat{i}-3r)&\equiv j'+a(j'+t_r)(m+h-{\rm lcm}(h,m))/2\nonumber\\
             &\equiv j'+(j'+t_r)(l/2+b/\gamma(m/2))\nonumber\\
             &\equiv j'+(j'+t_r)(l/2+b/\gamma(m/2)+bl/2\gamma(m/2))\nonumber\\
             &\equiv j+b(j+t_r)/\gamma(m/2)~({\rm mod}~b\gamma(l/2)/\gamma(m/2))\nonumber.
\end{align}

Suppose that \eqref{2.9} holds for some $r\in[b/3\gamma(m/2)]$ and $(i,j)\in\mathbb{Z}_m\times[l]$. If $\sigma_2(l)=1$, then $a(\hat{i}-3r)\equiv j~({\rm mod}~b/\gamma(m/2))$ and $j\equiv t_r~({\rm mod}~2)$, which imply that there are exactly $lm\gamma(m/2)/2b$ elements $(i,j)\in\mathbb{Z}_m\times[l]$ satisfying \eqref{2.9} for each $r\in[b/3\gamma(m/2)]$. Now we consider the case $\sigma_2(l)>1$. Since $\sigma_2(m)>1$, \eqref{2.9} is equivalence to $a(\hat{i}-3r)\equiv j+b(j+t_r)/2~({\rm mod}~b)$. It follows that $a(\hat{i}-3r)\equiv j~({\rm mod}~b/2)$. If $a(\hat{i}-3r)\equiv j~({\rm mod}~b)$, then $j\equiv t_r~({\rm mod}~2)$; if $a(\hat{i}-3r)\equiv j+b/2~({\rm mod}~b)$, then $j\equiv t_r+1~({\rm mod}~2)$. Hence, there are exactly $m/b$ elements $i\in\mathbb{Z}_m$ satisfying \eqref{2.9} for each $r\in[b/6]$ and $j\in[l]$. Since $|C|=ml/6$ from Step \ref{step-1}, $(i,j)\in C$ whenever \eqref{2.9} holds for some $r\in[b/3\gamma(m/2)]$.

\begin{step}
$C$ is a perfect code.
\end{step}

By Step \ref{step-2}, $C$ is an independent set. Let $(i,j)\in\mathbb{Z}_m\times[l]\setminus C$. By Step \ref{step-3}, \eqref{2.9} does not hold for all $r\in[b/3\gamma(m/2)]$. If $\sigma_2(l)=1$, then $a(\hat{i}-3r)\equiv j+c~({\rm mod}~b/\gamma(m/2))$ and $\hat{j}\equiv t_r+d~({\rm mod}~2)$ for some $r\in[b/3\gamma(m/2)]$ with $(c,d)\in\{(\pm1,0),(b/\gamma(m/2),1),(\pm(b/\gamma(m/2)+1),1)\}$. If $\sigma_2(l)>1$, then $\sigma_2(m)>1$, which implies $a(\hat{i}-3r)\equiv j+b/2(j+t_2)+c~({\rm mod}~b)$ and $j\equiv t_r+d~({\rm mod}~1)$ for some $r\in[b/6]$ with $(c,d)\in\{(\pm1,0),(b/2,1),(\pm(b/2+1),1)\}$. Thus,
\begin{align}
a(\hat{i}-3r)&\equiv j+b/\gamma(m/2)(j+t_r)+c~({\rm mod}~b\gamma(l/2)/\gamma(m/2)),\label{2.9-1}\\
j&\equiv t_r+d~({\rm mod}~2/\gamma(l/2))\label{2.9-2}
\end{align}
for some $r\in[b/3\gamma(m/2)]$ with $(c,d)\in\{(\pm1,0),(b/\gamma(m/2),1),(\pm(b/\gamma(m/2)+1),1)\}$.

If $(c,d)=(\pm1,0)$, from \eqref{2.9-1}, \eqref{2.9-2} and Step \ref{step-3}, then $(i-1,j)$ or $(i+1,j)$ is the unique neighbor of $(i,j)$ in $C$. If $(c,d)=(b/\gamma(m/2),1)$, from Step \ref{step-2}, then $(i+(m+h-{\rm lcm}(h,m))/2,j+l/2)$ is the unique neighbor of $(i,j)$ in $C$.

Now we consider the case $(c,d)=(\pm(b/\gamma(m/2)+1),1)$. If $c=-b/\gamma(m/2)-1$ and $j\neq0$, or $c=b/\gamma(m/2)+1$ and $j\neq l-1$, by \eqref{2.9-1} and \eqref{2.9-2}, then $a(\hat{i}-3r)\equiv (j\mp1)(b/\gamma(m/2)+1)+b t_r/\gamma(m/2)~({\rm mod}~b\gamma(l/2)/\gamma(m/2))$ and $j\equiv t_r+1~({\rm mod}~2/\gamma(l/2))$, which imply $(i,j-1)$ or $(i,j+1)$ is the unique neighbor of $(i,j)$ in $C$ from Step \ref{step-3}.

Suppose $c=-b/\gamma(m/2)-1$ and $j=0$. By \eqref{2.9-1} and \eqref{2.9-2}, if $\sigma_2(l)=1$, then $a(\hat{i}-3r)\equiv-1~({\rm mod}~b/\gamma(m/2))$ and $t_r\equiv1~({\rm mod}~2)$, which imply $a(\hat{i}+h-3r)\equiv l-1~({\rm mod}~b/\gamma(m/2))$ and $l-1\equiv t_r~({\rm mod}~2)$ since $ah\equiv l~({\rm mod}~b)$. If $\sigma_2(l)>1$, then $\sigma_2(m)>1$, and so $a(\hat{i}-3r)\equiv-b/2-1+bt_r/2~({\rm mod}~b)$, which imply
\begin{align}
a(\hat{i}+h-3r)\equiv-b/2-1+bt_r/2+l\equiv l-1+(l-1+t_r)b/2~({\rm mod}~b).\nonumber
\end{align}
Step \ref{step-3} implies that $(i+h,l-1)$ is the unique neighbor of $(i,0)$ in $C$.

Suppose $c=b/\gamma(m/2)+1$ and $j=l-1$. \eqref{2.9-1} and \eqref{2.9-2} imply $a(\hat{i}-3r)\equiv l+bt_r/\gamma(m/2)~({\rm mod}~b\gamma(l/2)/\gamma(m/2))$ and $t_r\equiv0~({\rm mod}~2/\gamma(l/2))$. Since $\sigma_2(m)\geq\sigma_2(l)$ and $ah\equiv l~({\rm mod}~b)$, one has $a(\hat{i}-h-3r)\equiv bt_r/\gamma(m/2)~({\rm mod}~b\gamma(l/2)/\gamma(m/2))$ and $t_r\equiv0~({\rm mod}~2/\gamma(l/2))$. Step \ref{step-3} implies that $(i-h,0)$ is the unique neighbor of $(i,l-1)$ in $C$.

This completes the proof of this proposition.
\end{proof}

\section{Proofs of Theorems \ref{main} and \ref{main2}}

In this section, we always assume that $\Gamma:=\Cay(G,S)$ admits a perfect code $D$, where $S$ is a generating set of $G$ with $|S|=5$. Define $N(v)$ to be the set of all neighbors of vertex $v$ in $\Gamma$.

To give the proofs of Theorems \ref{main} and \ref{main2}, we need several auxiliary lemmas.

\begin{lem}\label{1 involution}
The generating set $S$ has only one involution.
\end{lem}
\begin{proof}
Note that $S$ has at least one involution. If each element of $S$ is an involution, then $G$ is isomorphic to $\mathbb{Z}_2^3,\mathbb{Z}_2^4$ or $\mathbb{Z}_2^5$ since $\Gamma$ is connected, contrary to the fact that $6\mid |G|$ from \cite[Proposition 2.1]{Vu} and \cite[Lemma 2.3]{DYP2}. Thus, $S$ has exactly one or three involutions.

Assume the contrary, namely, $S:=\{\pm s,s_0,s_1,s_2\}$ has three distinct involutions $s_0,s_1,s_2$. Let $x(i,0,k)=is+ks_2$, $x(i,1,k)=x(i,0,k)+s_0$, $x(i,2,k)=x(i,1,k)+s_1$ and $x(i,3,k)=x(i,2,k)+s_0$ for $-1\leq i\leq 2$ and $0\leq k\leq1$. Note that $x(i,0,k)$ and $x(i,3,k)$ are adjacent for $-1\leq i\leq 2$ and $0\leq k\leq1$.

Since $\Gamma$ is vertex transitive, without loss of generality, we may assume that $x(0,0,0)\in D$. It follows that  $N(x(0,0,0))\cap D=\emptyset$, and so $x(0,3,0),x(0,0,1)\notin D$. Since $x(0,3,1),x(0,0,0)\in N(x(0,0,1))$, we get $x(0,3,1)\notin D$. The fact $N(x(0,3,1))=\{x(0,3,0),x(\pm1,3,1),x(0,0,1),x(0,2,1)\}$ implies $\{x(\pm1,3,1),x(0,2,1)\}\cap D\neq\emptyset$.

\textbf{Case 1.} $x(-1,3,1)$ or $x(1,3,1)\in D$.

Since the proofs are similar, we assume $x(1,3,1)\in D$. By $x(1,3,1),x(0,2,1)\in N(x(0,3,1))$, we get $x(0,2,1),x(0,3,1)\notin D$. The fact $x(1,3,1),x(1,1,1),x(1,2,0)\in N(x(1,2,1))$ implies $x(1,1,1),x(1,2,0),x(1,2,1)\notin D$. Since $x(0,0,0)\in D$ and $x(0,2,0),x(0,0,0),x(0,1,1)\in N(x(0,1,0))$, one has $x(0,2,0),x(0,1,1)\notin D$. Observe $N(x(0,2,1))=\{x(\pm1,2,1),x(0,1,1),x(0,3,1),x(0,2,0)\}$. The fact $x(0,2,1)\notin D$ implies $N(x(0,2,1))\cap D\neq\emptyset$, and so $x(-1,2,1)\in D$. It follows that $x(-1,1,1)\notin D$. Since $N(x(0,1,1))=\{x(\pm1,1,1),x(0,1,0),x(0,0,1),x(0,2,1)\}$ and $x(0,0,1)\notin D$, we get $x(0,1,0)\in D$, and so $N(x(0,0,0)\cap D\neq\emptyset$, contrary to the fact that $x(0,0,0)\in D$.

\textbf{Case 2.} $x(0,2,1)\in D$.

In view of the definition of a perfect code, we obtain $x(0,2,0),x(1,2,1)\notin D$. Since $x(1,2,0),x(0,2,1)\in N(x(1,2,1))$, one has $x(1,2,0)\notin D$. The fact that $x(0,0,0),x(1,1,0),x(1,3,0)\in N(x(1,0,0))$ implies $x(1,1,0),x(1,3,0),x(1,0,0)\notin D$. Note that $N(x(1,2,0))=\{x(0,2,0),x(2,2,0),x(1,1,0),x(1,3,0),x(1,2,1)\}$. Since $N(x(1,2,0))\cap D\neq\emptyset$, we get $x(2,2,0)\in D$, and so $x(2,1,0)\notin D$. The fact $x(0,2,1),x(1,1,1),x(0,1,0)\in N(x(0,1,1))$ implies $x(1,1,1),x(0,1,0)\notin D$. Since $N(x(1,1,0))=\{x(0,1,0),x(2,1,0),x(1,0,0),x(1,2,0),x(1,1,1)\}$, we obtain $D\cap(N(x(1,1,0))\cup x(1,1,0))=\emptyset$, a contradiction.
\end{proof}

By Lemma \ref{1 involution}, we may assume that $S=\{\pm s,\pm s',s_0\}$ with $o(s_0)=2$ and $o(s),o(s')>2$. Let $x(i,j,k)=is+js'+ks_0$ for all integers $i,j,k$, where the first coordinate could be read modulo $o(s)$, the second coordinate could be read modulo $o(s')$ and the last coordinate could be read modulo $2$. For all $x(i,j,k)\in V\Gamma$, we have
\begin{align}
N(x(i,j,k))=\{x(i\pm1,j,k),x(i,j\pm1,k),x(i,j,k+1)\}.\label{neighbors}
\end{align}

\begin{lem}\label{jb3}
If $x(i,j,k)\in D$, then $x(i+l,j+l,k+l)\in D$ for all $l\in\mathbb{Z}$ or $x(i-l,j+l,k+l)\in D$ for all $l\in\mathbb{Z}$.
\end{lem}
\begin{proof}
Since $\Gamma$ is vertex transitive, we may assume $(i,j,k)=(0,0,0)$. By the definition of a perfect code, we have $N(x(0,0,0))\cap D=\emptyset$, and so $x(0,1,0),x(0,0,1)\notin D$ from \eqref{neighbors}. Since $x(0,0,0),x(0,1,1)\in N(x(0,1,0))$, one gets $x(0,1,1)\notin D$. The fact $N(x(0,1,1))\cap D\neq\emptyset$ implies that $\{x(\pm1,1,1),(0,2,1)\}\cap D\neq\emptyset$.

Suppose $(0,2,1)\in D$. By the definition of a perfect code, one has $N(x(0,2,1))\cap D=\emptyset$, and so $x(1,2,1)\notin D$ from \eqref{neighbors}. Since $x(0,2,1),x(1,1,1)\in N(x(1,2,1))$, we get $x(1,1,1)\notin D$. The fact that $x(1,1,0),x(1,0,1),x(0,0,0)\in N(x(1,0,0))$ implies that $x(1,1,0),x(1,0,1)\notin D$. Since $x(0,1,1)\notin D$ and $N(x(1,1,1))\cap D\neq\emptyset$, one obtains $x(2,1,1)\in D$.

Since $x(2,1,1),x(2,2,0)\in N(x(2,2,1))$ and $x(1,2,0),x(0,2,1)\in N(x(0,2,0))$ from \eqref{neighbors}, we get $x(2,2,1),x(2,2,0),x(1,2,0),(0,2,0)\notin D$. The facts $N(x(1,2,0))\cap D\neq\emptyset$ and $x(1,1,0),x(1,2,1)\notin D$ imply $x(1,3,0)\in D$.

Since $x(1,3,0),x(-1,3,0)\in N(x(0,3,0))$ from \eqref{neighbors}, we have $x(-1,3,0)\notin D$. The fact $x(0,2,1),x(-1,2,0)\in N(x(-1,2,1))$ implies that $x(-1,2,0),x(-1,2,1)\notin D$. Since $x(0,0,0),x(-1,1,0)\in N(x(0,1,0))$, one obtains $x(-1,1,0)\notin D$. The facts $N(x(-1,2,0))\cap D\neq\emptyset$ and $x(0,2,0)\notin D$ imply $x(-2,2,0)\in D$.

By the definition of a perfect code, we have $N(x(-2,2,0))\cap D=\emptyset$, and so $x(-2,1,0)\notin D$ from \eqref{neighbors}. Since $x(-1,1,1),x(0,2,1)\in N(x(-1,2,1))$, one gets $x(-1,1,1)\notin D$. The fact $x(-1,0,0)\in N(x(0,0,0))$ implies $x(-1,0,0)\notin D$. Note that $x(-1,1,0)\notin D$. Since $x(0,1,0),x(-1,2,0)\notin D$, we get $N(x(-1,1,0))\cap D=\emptyset$, a contradicition.

Note that $x(1,1,1)\in D$ or $x(-1,1,1)\in D$. Similarly, $\{x(0,2,0),x(2,2,2)\}\cap D\neq\emptyset$ or $\{x(0,2,0),x(-2,2,2)\}\cap D\neq\emptyset$. Since $x(0,2,0),x(0,0,0)\in N(x(0,1,0))$, we have $x(2,2,2)\in D$ or $x(-2,2,2)\in D$. By induction, one gets $x(l,l,l)\in D$ for all $l\in\mathbb{Z}$ or $x(-l,l,l)\in D$ for all $l\in\mathbb{Z}$. This completes the proof of this lemma.
\end{proof}

\begin{lem}\label{2|mn}
We have $2\mid o(s)o(s')$.
\end{lem}
\begin{proof}
Since $\Gamma$ is vertex transitive, we may assume $x(0,0,0)\in D$. By Lemma \ref{jb3}, one has $x(-o(s)o(s'),o(s)o(s'),o(s)o(s'))=x(o(s)o(s'),o(s)o(s'),o(s)o(s'))=x(0,0,o(s)o(s'))\in D$. The fact $x(0,0,1)\in N(x(0,0,0))$ implies $x(0,0,1)\notin D$, and so $2\mid o(s)o(s)'$.
\end{proof}

\begin{lem}\label{jb4}
Let $x(i,j,k)\in D$. The following hold:
\begin{itemize}
\item[{\rm (i)}] $\{x(i,j,k+1),x(i-1,j,k),x(i-2,j,k),x(i-1,j,k+1),x(i-2,j,k+1)\}\cap D=\emptyset$, and $x(i-3,j,k)$ or $x(i-3,j,k+1)\in D$;

\item[{\rm (ii)}] $\{x(i,j,k+1),x(i,j-1,k),x(i,j-2,k),x(i,j-1,k+1),x(i,j-2,k+1)\}\cap D=\emptyset$, and $x(i,j-3,k)$ or $x(i,j-3,k+1)\in D$.
\end{itemize}
\end{lem}
\begin{proof}
(i) Since $\Gamma$ is vertex transitive, we may assume $(i,j,k)=(0,0,0)$. It follows from \eqref{neighbors} that $x(0,0,1),x(-1,0,0)\notin D$. Since $x(0,0,0),x(-2,0,0),x(-1,0,1)\in N(x(-1,0,0))$, we get $x(-2,0,0),x(-1,0,1)\notin D$.

By Lemma \ref{jb3}, we have $x(al,l,l)\in D$ for each integer $l$ and some $a\in\{\pm1\}$. In view of Lemma \ref{2|mn}, we get
$x(-1,-a,1)=x(o(s)o(s')-1,a(o(s)o(s')-1),a(o(s)o(s')-1))\in D.$ Since $x(-1,-a,1),x(-2,0,1)\in N(x(-2,-a,1))$ from \eqref{neighbors}, one obtains $x(-2,0,1)\notin D$. Then $x(0,0,1),x(-1,0,0),x(-2,0,0),x(-1,0,1),x(-2,0,1)\notin D$.

Since $x(-1,-a,1)\in D$ and $x(-2,-a,0),x(-1,-a,1)\in N(x(-2,-a,1))$ from \eqref{neighbors}, we have $x(-2,-a,0)\notin D$. The fact $x(-2,0,0)\notin D$ implies $N(x(-2,0,0))\cap D\neq\emptyset$. Since $(-1,0,0),x(-2,0,1)\notin  D$, one gets $x(-3,0,0)$ or $x(-2,a,0)\in D$.

Suppose $x(-3,0,0)\notin D$. Note that $x(-2,a,0)\in D$. It follows from \eqref{neighbors} that $x(-2,a,1)\notin D$. Since $x(-1,-a,1)\in D$, we have $x(-1,0,1),x(-2,-a,1)\notin D$. The fact $x(-2,0,1)\notin D$ implies $N(x(-2,0,1))\cap D\neq\emptyset$. Since $x(-2,0,0)\notin D$, one gets $x(-3,0,1)\in D$.

This completes the proof of (i).

(ii) The proof is similar, hence omitted.
\end{proof}

\begin{lem}\label{tongyi}
The one of the following holds:
\begin{itemize}
\item[{\rm(i)}] $x(i+1,j+1,k+1)\in D$ for all $x(i,j,k)\in D$;

\item[{\rm(ii)}] $x(i-1,j+1,k+1)\in D$ for all $x(i,j,k)\in D$.
\end{itemize}
\end{lem}
\begin{proof}
Since $\Gamma$ is vertex transitive, we may assume $x(0,0,0)\in D$. By Lemma \ref{jb3}, we have $x(i+l,j+l,k+l)\in D$ with $l\in\mathbb{Z}$ or $x(i-l,j+l,k+l)\in D$ with $l\in\mathbb{Z}$ for each $x(i,j,k)\in D$.  It suffices to show that $x(l,j+l,k+l)\in D$ for all $x(0,j,k)\in D$, or $x(-l,j+l,k+l)\in D$ for all $x(0,j,k)\in D$.

Suppose not. By Lemma \ref{jb4}, we may assume $x(l,l,l),x(3-l,l,k'+l)\in D$ for all $l\in\mathbb{Z}$ and some $k'\in\{0,1\}$. Then $x(2,2,0),x(2,1,k'+1)\in D$. Since $x(2,1,0)\in N(x(2,2,0))$ from \eqref{neighbors}, one gets $x(2,1,0)\notin D$, and so $k'=0$, which imply $x(2,1,1),x(2,2,0)\in N(x(2,1,0))\cap D$, a contradiction.

Thus, (i) or (ii) is valid.
\end{proof}

\begin{lem}\label{jb5}
Let $a\in\{\pm1\}$ and $x(0,0,0)=x(h,l,0)$ for some $(h,l)\in\{(i,j)\mid 0\leq i\leq o(s)~\textrm{and}~0\leq j\leq o(s')\}$. Suppose $x(i+a,j+1,k+1)\in D$ for all $x(i,j,k)\in D$. Then $3\mid(l-ah)$. Moreover, if $\sigma_2(h)=\sigma_2(l)=0$, then $\sigma_2(o(s))>\sigma_2(l-ah)$ and $\sigma_2(o(s'))>\sigma_2(l-ah)$.
\end{lem}
\begin{proof}
Since $\Gamma$ is vertex transitive, we may assume $x(ai,i,i)\in D$ for each integer $i$. The fact $x(0,0,0)=x(h,l,0)$ implies $x(i,j,k)=x(i-h,j-l,k)$ for all $x(i,j,k)\in V\Gamma$.

Suppose $3\nmid(l-ah)$. It follows that $3\mid(al-h+b)$ for some $b\in\{\pm1\}$. Note that $x(al-h,0,l)=x(al,l,l)\in D$. Since $x(0,0,0)\in D$, from Lemma \ref{jb4} (i), we have $\{x(al-h+b,0,l),x(al-h+b,0,l+1)\}\cap D\neq\emptyset$. Since $D$ is an independent set, from \eqref{neighbors}, one gets $x(al-h+b,0,l+1)\in D$, which implies $|N(x(al-h,0,l+1))\cap D|>1$, a contradiction. The first statement is valid.

Now suppose $2\nmid l$ and $2\nmid h$. If $\sigma_2(o(s))\leq\sigma_2(al-h)$, then
\begin{align}
x(0,0,1)=x\Big(\frac{o(s)(al-h)}{2^{\sigma_2(o(s))}},0,1\Big)=x\Big(\frac{ao(s)l}{2^{\sigma_2(o(s))}},\frac{o(s)l}{2^{\sigma_2(o(s))}},\frac{o(s)l}{2^{\sigma_2(o(s))}}\Big)\in D,\nonumber
\end{align}
contrary to the fact that $x(0,0,0)\in D\cap N(x(0,0,1))$. Hence, $\sigma_2(o(s))>\sigma_2(al-h)$. Similarly, $\sigma_2(o(s'))>\sigma_2(al-h)$. The second statement is also valid.
\end{proof}

\begin{lem}\label{s_0=ms/2}
Let $a\in\{\pm1\}$ and $x(0,0,0)=x(h,l,0)$ for some $(h,l)\in\{(i,j)\mid 0\leq i\leq o(s)~\textrm{and}~0\leq j\leq o(s')\}$. Suppose $x(i+a,j+1,k+1)\in D$ for all $x(i,j,k)\in D$. Then the following hold:
\begin{itemize}
\item[{\rm(i)}] if $s_0=o(s)s/2$, then $\sigma_2(l)=0$ and $\sigma_2(o(s))=\sigma_2(l-ah)+1$, or $\sigma_2(h)\sigma_2(l)\neq0$ and $\sigma_2(o(s))\leq\sigma_2(l-ah)$;

\item[{\rm(ii)}] if $s_0=o(s')s'/2$, then $\sigma_2(h)=0$ and $\sigma_2(o(s'))=\sigma_2(l-ah)+1$, or $\sigma_2(h)\sigma_2(l)\neq0$ and $\sigma_2(o(s'))\leq\sigma_2(l-ah)$.
\end{itemize}
\end{lem}
\begin{proof}
Since $\Gamma$ is vertex transitive, we may assume $x(ai,i,i)\in D$ for each integer $i$.

(i) Since $s_0=o(s)s/2$, we have $2\mid o(s)$. Note that $x(i,j,k)=x(i-h,j-l,k)$ and $x(i,j,1)=x(i+o(s)/2,j,0)$ for all integers $i,j,k$. Assume the contrary, namely, $\sigma_2(l)=0$ and $\sigma_2(o(s))\neq\sigma_2(l-ah)+1$, $\sigma_2(l)\neq0$ and $\sigma_2(h)=0$, or $\sigma_2(h)\sigma_2(l)\neq0$ and $\sigma_2(o(s))>\sigma_2(l-ah)$. If $\sigma_2(l)=0$, from Lemma \ref{jb5}, then $\sigma_2(o(s))>\sigma_2(l-ah)+1$; if $\sigma_2(l)\neq0$ and $\sigma_2(h)=0$, then $\sigma_2(l-ah)=0$, and so $\sigma_2(o(s))>\sigma_2(l-ah)$. Observe
\begin{align}
x\Big(\frac{ao(s)l}{2^{\sigma_2(l-ah)+1}},\frac{o(s)l}{2^{\sigma_2(l-ah)+1}},\frac{o(s)l}{2^{\sigma_2(l-ah)+1}}\Big)=x\Big(\frac{o(s)(al-h)}{2^{\sigma_2(l-ah)+1}},0,0\Big)=x(o(s)/2,0,0).\nonumber
\end{align}
It follows that $x(0,0,1)\in D$, contrary to the fact that $x(0,0,0)\in D\cap N(x(0,0,1))$. This proves (i).

(ii) The proof is similar, hence omitted.
\end{proof}

\begin{lem}\label{s_0=ms/2+ns'/2}
Let $l$ be the minimal positive integer with $x(0,0,0)=x(h,l,0)$ for some $h\in[o(s)]$. Suppose $s_0=o(s)s/2+o(s')s'/2$. Then $\sigma_2(o(s))\neq\sigma_2(o(s)')$ or $\sigma_2(o(s))=\sigma_2(o(s)')=1$. Moreover, the following hold:
\begin{itemize}
\item[{\rm (i)}] if $\sigma_2(o(s))=\sigma_2(o(s)')=1$, then $\sigma_2(h)\geq\sigma_2(o(s))=\sigma_2(l)=1$;

\item[{\rm (ii)}] if $\sigma_2(o(s))>\sigma_2(o(s)')$, then $\sigma_2(h)\geq\sigma_2(o(s))>\sigma_2(l)\geq1$;

\item[{\rm (iii)}] if $\sigma_2(o(s'))>\sigma_2(o(s))$, then $\sigma_2(l)\geq\sigma_2(o(s'))>\sigma_2(h)\geq1$.
\end{itemize}
\end{lem}
\begin{proof}
Since $\Gamma$ is vertex transitive, from Lemma \ref{tongyi}, we may assume $x(ai,i,i)\in D$ for each $i\in\mathbb{Z}$ and some $a\in\{\pm1\}$. Since $s_0=o(s)s/2+o(s')s'/2$, we have $2\mid o(s)$ and $2\mid o(s')$. Note that $x(i,j,1)=x(i+o(s)/2,j+o(s')/2,0)$ for all integers $i,j$. If $\sigma_2(o(s))=\sigma_2(o(s'))>1$, then
\begin{align}
 x\Big(\frac{ao(s)o(s')}{2^{\sigma_2(o(s))+1}},\frac{o(s)o(s')}{2^{\sigma_2(o(s))+1}},\frac{o(s)o(s')}{2^{\sigma_2(o(s))+1}}\Big)=x\Big(\frac{o(s)}{2},\frac{o(s')}{2},0\Big)=x(0,0,1)\in D,\nonumber
\end{align}
a contradiction. Thus, the first statement is valid.

Suppose that $\sigma_2(o(s))>\sigma_2(o(s)')$ or $\sigma_2(o(s))=\sigma_2(o(s)')=1$. Note that $x(i,j,k)=x(i-h,j-l,k)$ for all integers $i,j,k$. If $h=0$, then $l=o(s')$, which implies that (i) and (ii) hold for this case. Now we consider the case $h\neq0$. Since $x(0,l,0)=x(-h,0,0)$ and the order of $x(-h,0,0)$ in $\langle s\rangle$ is $o(s)/{\rm gcd}(o(s),h)$, one has $o(s')=o(s)l/{\rm gcd}(o(s),h)={\rm lcm}(o(s),h)l/h$. If $\sigma_2(o(s))>\sigma_2(h)$, then
\begin{align}
s_0=x\Big(\frac{o(s)}{2},\frac{o(s')}{2},0\Big)=x\Big(\frac{o(s)}{2},\frac{{\rm lcm}(o(s),h)l}{2h},0\Big)=x\Big(\frac{o(s)-{\rm lcm}(o(s),h)}{2},0,0\Big),\nonumber
\end{align}
which implies $s_0=0$, a contradiction. Hence, $\sigma_2(h)\geq\sigma_2(o(s))$. It follows that $\sigma_2(o(s'))=\sigma_2(l)$. Thus, (i) and (ii) hold.

Similarly, (iii) is also valid.
\end{proof}

Now we are ready to give a proof of Theorem \ref{main}.

\begin{proof}[Proof of Theorem~\ref{main}]
The proof of the sufficiency is straightforward by Propositions \ref{cayley},~\ref{(iv)},~\ref{(v)},~\ref{(vi)} and Corollaries \ref{cayley-2},~\ref{cayley-3}.

Now we prove the necessity. Since  $\Gamma$ admits a perfect code and $x(o(s),0,0)=x(0,o(s'),0)=x(0,0,0)$, from Lemmas \ref{tongyi} and \ref{jb5}, we get $3\mid o(s)$ and $3\mid o(s')$.

Let $l$ be the minimal positive integer with $x(0,0,0)=x(h,l,0)$ for some $h\in[o(s)]$. By Lemmas \ref{tongyi} and \ref{jb5}, we have $3\mid(l-ah)$ for some $a\in\{\pm1\}$.

\textbf{Case 1.} $s_0\notin\langle s,s'\rangle$.

By Lemma \ref{2|mn}, we may assume $2\mid o(s)$. Since $x(i,0,0)=x(i+h,l,0)$ for all $i$, from \eqref{neighbors} and Construction \ref{construcution}, $\Gamma$ is isomorphic to $\Gamma_{o(s),l,h}\times K_2$. In view of Lemmas \ref{tongyi} and \ref{jb5}, we have $\sigma_2(h)\sigma_2(l)\neq0$ or $\sigma_2(o(s))>\sigma_2(l-ah)$. By Proposition \ref{cayley}, $\Gamma$ is a Cayley graph on an abelian group.

\textbf{Case 2.} $s_0\in\langle s,s'\rangle$.

We divide our proof into two subcases according to whether $s_0$ is in the set $\{(o(s)/2)s,(o(s')/2)s'\}$.

\textbf{Case 2.1.} $s_0\in\{(o(s)/2)s,(o(s')/2)s'\}$.

Without loss of generality, we may assume $s_0=(o(s)/2)s$. It follows that $6\mid o(s)$ and $x(i,j,1)=x(i+o(s)/2,j,0)$ for all integers $i,j$. Since $x(i,0,0)=x(i+h,l,0)$ for all $i$, from \eqref{neighbors} and Construction \ref{cons-2}, $\Gamma$ is isomorphic to $\Gamma_{o(s),l,h}'$. By Lemma \ref{tongyi} and Lemma \ref{s_0=ms/2} (i), one gets $\sigma_2(l)=0$ and $\sigma_2(o(s))=\sigma_2(l-ah)+1$, or $\sigma_2(h)\sigma_2(l)\neq0$ and $\sigma_2(o(s))\leq\sigma_2(l-ah)$. It follows from Corollary \ref{cayley-2} that $\Gamma$ is a Cayley graph on an abelian group.

\textbf{Case 2.2.} $s_0\notin\{(o(s)/2)s,(o(s')/2)s'\}$.

Without loss of generality, we may assume $\sigma_2(o(s))\geq\sigma_2(o(s'))$. If $2\nmid o(s')$, from Lemma \ref{2|mn}, then $2\mid o(s)$, and so $(o(s)/2)s$ is the unique involution in $G$, contrary to the fact that $s_0\neq(o(s)/2)s$. Then $2\mid o(s)$ and $2\mid o(s')$. Note that $(o(s)/2)s$, $(o(s')/2)s'$ and $(o(s)/2)s+(o(s')/2)s'$ are all involutions in $G$. It follows that $s_0=(o(s)/2)s+(o(s')/2)s'$, and so $x(i,j,1)=x(i+o(s)/2,j+o(s')/2,0)$ for all integers $i,j$.

By Lemma \ref{s_0=ms/2+ns'/2}, we get $\sigma_2(h)\geq\sigma_2(o(s))>\sigma_2(l)\geq1$ or $\sigma_2(h)\geq\sigma_2(o(s))=\sigma_2(l)=1$. If $h=0$, then $o(s')=l$, and so $x(i,j,1)=x(i+o(s)/2,j+l/2,0)$ for all integers $i,j$. Now suppose $h\neq0$. Since $x(0,l,0)=x(-h,0,0)$ and the order of $x(-h,0,0)$ in $\langle s\rangle$ is $o(s)/{\rm gcd}(o(s),h)$, we have $o(s')=o(s)l/{\rm gcd}(o(s),h)={\rm lcm}(o(s),h)l/h$. It follows that $x(i,j,1)=x(i+o(s)/2,j+{\rm lcm}(o(s),h)l/2h,0)=x(i+(o(s)+h-{\rm lcm}(o(s),h))/2,j+l/2,0)$.

Since $x(i,0,0)=x(i+h,l,0)$ for all $i$, from \eqref{neighbors} and Construction \ref{cons-3}, $\Gamma$ is isomorphic to $\Gamma_{o(s),l,h}''$. By Corollary \ref{cayley-3}, $\Gamma$ is a Cayley graph on an abelian group.

This completes the proof of the necessity.
\end{proof}

Next, we give a proof of Theorem \ref{main2}.

\begin{proof}[Proof of Theorem~\ref{main2}]
Note that $x(0,0,0)\in D$. By Lemma \ref{tongyi}, there exists $a\in\{\pm1\}$ such that $x(i+a,j+1,k+1)\in D$ for $x(i,j,k)\in D$. Let $l$ be the minimal positive integer with $x(0,l,0)\in\langle s\rangle$. Then there exists $h\in[o(s)]$ such that $x(h,l,0)=x(0,0,0)$. It follows that $x(i,j,k)=x(i-h,j-l,k)$ for all $i,j,k\in\mathbb{Z}$.

Let $H$ be the subset of $\langle s\rangle$ such that $is\in H$ whenever there exists integer $j$ satisfying $x(i,0,k)=x(aj,j,j)\in D$ for some $k\in\{0,1\}$. We claim that $H$ is a subgroup of $\langle s\rangle$. Since $0\in H$, $H$ is not empty. Let $is,i's\in H$. Then there exist integers $j,j'$ such that $x(i,0,k)=x(aj,j,j)\in D$ and $x(i',0,k')=x(aj',j',j')\in D$ for some $k,k'\in\{0,1\}$. Note that $x(-aj',-j',-j')=x(-i',0,-k')$. It follows that $x(i-i',0,k-k')=x(a(j-j'),j-j',j-j')\in D$ and so $(i-i')s\in H$. Thus, our claim is valid.

Suppose that $\Gamma$ is isomorphic to the graph in Theorem \ref{main} (i). Note that $is\in H$ if and only if there exists integer $j$ satisfying $l\mid j$ and $x(aj,j,j)=x(i,0,k)$ for some $k\in\{0,1\}$. Since $x(al,l,l)=x(al-h,0,0)$ or $x(al-h,0,1)$, from the claim, one has $H=\langle(l-ah)s\rangle$. Since $H$ has ${\rm gcd}(ah-l,m)$ cosets in $\langle s\rangle$, from Lemma \ref{jb4} (i) and Lemma \ref{tongyi}, we get $x(3r+aj,j,j+t_r)\in D$ with $t_r\in\{0,1\}$ for all $j\in\mathbb{Z}$ and $1\leq r<{\rm gcd}(ah-l,m)/3-1$. It follows from Proposition \ref{(iv)} that (i) is valid.

Suppose that $\Gamma$ is isomorphic to the graph in Theorem \ref{main} (ii). Note that $is\in H$ if and only if there exists integer $j$ satisfying $l\mid j$ and $x(aj,j,j)=x(i,0,k)$ for some $k\in\{0,1\}$. Since $x(al,l,l)=x(al-h,0,0)$ or $x(al-h,0,1)$, from the claim, one has $H=\langle(l-ah)s\rangle$. Since $x(i,j,1)=x(i+m/2,j,0)$ for all integers $i,j$ from Construction \ref{cons-2}, at most one element of $x(i,j,0)+\langle ms/2\rangle$ belongs to $D$. Note that the subgroup $\langle (l-ah)s+\langle ms/2\rangle\rangle$ has ${\rm gcd}(ah-l,m/2)$ cosets in the quotient group $\langle s\rangle/\langle ms/2\rangle$. By Lemma \ref{jb4} (i) and Lemma \ref{tongyi}, we have $x(3r+aj+(j+t_r)m/2,j,0)=x(3r+aj+t_rm/2,j,j)\in D$ with $t_r\in\{0,1\}$ for all $1\leq r\leq {\rm gcd}(ah-l,m/2)/3-1$. If $\sigma_2(l)=0$ and $\sigma_2(m)=\sigma_2(l-ah)+1$, then ${\rm gcd}(ah-l,m/2)={\rm gcd}(ah-l,m)$, which implies that (i) is valid from Proposition \ref{(v)}; if $\sigma_2(h)\sigma_2(l)\neq0$ and $\sigma_2(m)\leq\sigma_2(l-ah)$, then ${\rm gcd}(ah-l,m/2)={\rm gcd}(ah-l,m)/2$, which implies that (ii) is valid.

We only need to consider the case that $\Gamma$ is isomorphic to the graph in Theorem \ref{main} (iii). By Construction \ref{cons-3}, we have $x(i,j,1)=x(i+(h+m-{\rm lcm}(m,h))/2,j+l/2,0)$ for all integers $i,j$.

\textbf{Case 1.} $\sigma_2(l)>1$.

Note that $x(i+al/2,l/2,l/2)=x(i+al/2,l/2,0)$ for all integers $i,j$. It follows that $is\in H$ if and only if there exists integer $j$ satisfying $l\mid j$ and $x(aj,j,j)=x(i,0,0)$. Since $x(al,l,l)=x(al-h,0,0)$, from the claim, one has $H=\langle(l-ah)s\rangle$.

By Construction \ref{cons-3}, we have $x((al-h)/2,0,1)=x((al+m-{\rm lcm}(h,m))/2,l/2,0)$. If $\sigma_2(h)=\sigma_2(m)$, then $x((al-h)/2,0,1)=x(al/2,l/2,0)=x(al/2,l/2,l/2)\in D$. Now suppose $\sigma_2(h)>\sigma_2(m)$. Then $x((al-h)/2,0,1)=x((al+m)/2,l/2,0)$. Since $\sigma_2(h)>\sigma_2(l)$, we have $\sigma_2(ah-l)=\sigma_2(l)$. The fact $\sigma_2(m)>\sigma_2(l)$ implies that there exists an integer $k$ such that $k\cdot x(al-h,0,0)=x(m/2,0,0)$. It follows that $x((al-h)/2,0,1)=x(al/2+akl-kh,l/2,0)=x(a(k+1/2)l,(k+1/2)l,0)=x(a(k+1/2)l,(k+1/2)l,(k+1/2)l)\in D$. We conclude that $x(i+(al-h)/2,j,k+1)=x(i,j,l)+x(anl/2,nl/2,nl/2)\in D$ for some $n\in\mathbb{Z}$ with $x(i,j,k)\in D$.

Since $H$ has ${\rm gcd}(ah-l,m)$ cosets in $\langle s\rangle$, from Lemma \ref{jb4} (i) and Lemma \ref{tongyi}, we have $x(3r+aj,j,j+t_r)\subseteq D$ with $t_r\in\{0,1\}$ for $1\leq r\leq {\rm gcd}((ah-l),m)/6-1$. By Proposition \ref{(vi)}, (ii) is valid.

\textbf{Case 2.} $\sigma_2(h)\geq\sigma_2(m)>\sigma_2(l)=1$ or $\sigma_2(h)=\sigma_2(m)=\sigma_2(l)=1$.

For each integer $k$, if $2\mid k$, then $x(akl/2,kl/2,kl/2)=x(k(al-h)/2,0,0)$; if $2\nmid k$ and $\sigma_2(h)=\sigma_2(m)$, then $x(akl/2,kl/2,kl/2)=x(akl/2,kl/2,1)=x((akl+h)/2,(k+1)l/2,0)=x(k(al-h)/2,0,0)$; if $2\nmid k$ and $\sigma_2(h)>\sigma_2(m)$, then $\sigma_2(m)>\sigma_2(ah-l)$, and so $i\cdot x(al-h,0,0)=x(m/2,0,0)$ for some integer $i$, which imply that $x(a(k+2i)l/2,(k+2i)l/2,(k+2i)l/2)=x(a(k+2i)l/2,(k+2i)l/2,1)=x((a(k+2i)l+h+m)/2,(k+2i+1)l/2,0)=x(k(al-h)/2,0,0)$.
It follows that $is\in H$ if and only if there exists integer $j$ satisfying $l/2\mid j$ and $x(aj,j,j)=x(i,0,0)$. The claim implies $H=\langle(l-ah)s/2\rangle$.

Since $H$ has ${\rm gcd}((ah-l)/2,m)$ cosets in $\langle s\rangle$, from Lemma \ref{jb4} (i) and Lemma \ref{tongyi}, we have $x(3r+aj,j,j+t_r)\subseteq D$ with $t_r\in\{0,1\}$ for $1\leq r\leq {\rm gcd}((ah-l)/2,m)/3-1$. If $\sigma_2(m)>1$, then $\sigma_2(m)>\sigma_2(ah-l)$, and so ${\rm gcd}((ah-l)/2,m)={\rm gcd}(ah-l,m)/2$, which imply that (ii) is valid from Proposition \ref{(vi)}; if $\sigma_2(m)=1$, then $\sigma_2(m)<\sigma_2(ah-l)$, and so ${\rm gcd}((ah-l)/2,m)={\rm gcd}(ah-l,m)$, which imply that (i) is valid from Proposition \ref{(vi)}.

\textbf{Case 3.} $\sigma_2(h)>\sigma_2(m)=\sigma_2(l)=1$.

For each integer $k$, if $2\mid k$, then $x(akl/2,kl/2,kl/2)=k/2\cdot x((al-h),0,0)$; if $2\nmid k$, from $\sigma_2(h)>\sigma_2(m)$, then $\sigma_2(m)=\sigma_2(ah-l)$ and $i\cdot x(al-h,0,0)=x(0,0,0)$ for some odd integer $i$, which imply $x(akl/2,kl/2,kl/2)=x(akl/2,kl/2,1)=x((akl+m+h)/2,(k+1)l/2,0)=x((k(al-h)+m)/2,0,0)=(k+i)/2\cdot x(al-h,0,0)$. It follows that $is\in H$ if and only if there exists integer $j$ satisfying $l/2\mid j$ and $x(aj,j,j)=x(i,0,0)$. In view of the claim, one gets $H=\langle(l-ah)s\rangle$.

Since $H$ has ${\rm gcd}(ah-l,m)$ cosets in $\langle s\rangle$, from Lemma \ref{jb4} (i) and Lemma \ref{tongyi}, we have $x(3r+aj,j,j+t_r)\subseteq D$ with $t_r\in\{0,1\}$ for $1\leq r\leq {\rm gcd}(ah-l,m)/3-1$. It follows from Proposition \ref{(vi)} that (i) is valid.
\end{proof}

\bigskip
\noindent \textbf{Acknowledgements}~~
Y. Yang is supported by NSFC (12101575) and the Fundamental Research
Funds for the Central Universities (2652019319). X. Ma is supported by the National Natural Science Foundation of China
(11801441, 12326333), and Shaanxi Fundamental Science Research Project for Mathematics and Physics
(Grant No. 22JSQ024).

\section*{Data Availability Statement}

Data sharing not applicable to this article as no datasets were generated or analysed during the current study.

\end{document}